\documentclass[11pt]{smfart}
\usepackage{color}
\usepackage{amssymb,verbatim}
\usepackage{hyperref}
\usepackage{amsthm,array,amssymb,amscd,amsfonts,amssymb,latexsym, url}
\usepackage{amsmath}
\usepackage[all]{xy}
\usepackage[french]{babel}

\newcounter{spec}
{\end{list}}

\renewcommand{\P}{{\mathbf P}}

\newcommand{\Z}{{\mathbb Z}}
\newcommand{\Q}{{\mathbb Q}}

\renewcommand{\L}{{\mathbb L}}

\newcommand{\Br}{{\operatorname{Br   }}}

\newcommand{\Hom}{{\operatorname{Hom}}}

\newcommand{\Res}{{\operatorname{Res}}}
\newcommand{\Spec}{{\operatorname{Spec \ }}}

\newcommand{\Norm}{{\operatorname{Norm}}}

\renewcommand{\lim}{\varprojlim}

\numberwithin{equation}{section}

\newfont{\gothic}{eufb10}

\renewcommand{\qed}{{\hfill$\square$}}

\newtheorem{theo}{Th\'{e}or\`{e}me}[section]
\newtheorem{prop}[theo]{Proposition}

\newtheorem{lem}[theo]{Lemme}
\newtheorem{cor}[theo]{Corollaire}
\theoremstyle{definition}
\newtheorem{defi}[theo]{D\'efinition}
\theoremstyle{remark}
\newtheorem{rema}[theo]{Remarque}

\newcommand{\bthe}{\begin{theo}}
\newcommand{\ble}{\begin{lem}}
\newcommand{\bpr}{\begin{prop}}
\newcommand{\bco}{\begin{cor}}
\newcommand{\bde}{\begin{defi}}
\newcommand{\ethe}{\end{theo}}
\newcommand{\ele}{\end{lem}}
\newcommand{\epr}{\end{prop}}
\newcommand{\eco}{\end{cor}}
\newcommand{\ede}{\end{defi}}

\newcommand{\Pic}{\operatorname{Pic}}

\newcommand{\G}{{\mathbb G}}

\def\X{{\overline X}}
\def\k{{\overline k}}
\DeclareFontFamily{U}{wncy}{}
\DeclareFontShape{U}{wncy}{m}{n}{%
<5>wncyr5%
<6>wncyr6%
<7>wncyr7%
<8>wncyr8%
<9>wncyr9%
<10>wncyr10%
<11>wncyr10%
<12>wncyr6%
<14>wncyr7%
<17>wncyr8%
<20>wncyr10%
<25>wncyr10}{}
\DeclareMathAlphabet{\cyr}{U}{wncy}{m}{n}
\begin{document}

  \title[Droites sur les surfaces cubiques ] {Normes de droites sur les surfaces cubiques}

\author{J.-L. Colliot-Th\'el\`ene}
\address{Laboratoire de Math\'ematiques d'Orsay, Univ. Paris-Sud, CNRS, Universit\'e Paris-Saclay, 91405 Orsay, France}
\email{jlct@math.u-psud.fr}

\author{D. Loughran}
\address{School of Mathematics\\ University of Manchester\\ Oxford Road\\ Manchester\\ M13 9PL\\ United Kingdom}
\email{daniel.loughran@manchester.ac.uk}

\date{soumis le 19 septembre 2017; r\'evis\'e le 21 janvier 2018; accept\'e le 9 mai 2018}
\maketitle

 \begin{altabstract}
 Let $k$ be a field and $X \subset \P^3_{k}$ a smooth cubic  surface.
Let $\Delta=\Delta(X)  \subset \Pic(X)$ be the  
 subgroup spanned
by norms to $k$ of $K$-lines on $X_{K}=X \times_{k}K$ for $K$ running through the finite
separable extensions of $k$. The quotient $\Pic(X)/\Delta$
is a finite, 3-primary group. If $X$ contains a line defined over $k$, then $\Delta=\Pic(X)$.
  \end{altabstract}

\section{Introduction}

Soient $k$ un corps
et  $X \subset \P^3_{k}$ une $k$-surface cubique projective lisse. 
 Soit $\k$ une cl\^oture s\'eparable de $k$. Notons   $\X=X\times_{k}\k$.
  Soit $\Delta=\Delta(X) \subset \Pic(X)$ le sous-groupe 
  engendr\'e par toutes les classes de diviseurs $\Norm_{K/k}(L)$ dans $\Pic(X)$,
  pour $K/k$ parcourant les extensions finies s\'eparables de $k$ et, pour
  $K$ donn\'e, $L$ parcourant les droites de $X_{K} = X_{K}=X \times_{k}K \subset \P^3_{K}$
  qui sont d\'efinies sur $K$.
Un r\'esultat classique
 dit que le groupe de Picard $\Pic(\X)$ est engendr\'e par les droites de $\X$.

 \begin{theo}\label{thmgeneral}
 Soient $k$ un corps et  $X\subset \P^3_{k}$ une $k$-surface cubique lisse.
Soit $\Delta \subset \Pic(X)$ le sous-groupe  engendr\'e par
les classes de normes de droites sur $X_{K}$ pour $K/k$
parcourant les extensions finies de $k$.

(i) Le quotient $\Pic(X)/\Delta$ est un groupe fini 3-primaire.

(ii) Si $X$ poss\`ede une droite d\'efinie sur $k$, alors $\Delta=\Pic(X)$.
\end{theo}
 
 Ce th\'eor\`eme est \'etabli au paragraphe \S \ref{toutescubiques}.
 Pour \'etablir ce r\'esultat, on commence par \'etudier (\S \ref{cubiquesavecdroite})
  le cas des surfaces cubiques lisses poss\'edant une droite rationnelle.
 \`A toute telle droite est associ\'ee une fibration en coniques.
 Lorsque la fibration admet une section, on montre qu'il existe une
 section donn\'ee par une droite (Th\'eor\`eme \ref{sectiondroite}).
 
Le th\'eor\`eme \ref{thmgeneral} (ii) est un analogue en dimension 2 d'un r\'esultat de M.~Shen  \cite[Thm. 1.7]{shen} sur
les  hypersurfaces cubiques de dimension au moins 3
poss\'edant une droite d\'efinie sur $k$, 
r\'esultat qui nous a \'et\'e signal\'e par O. Wittenberg.

On utilise librement les propri\'et\'es des surfaces cubiques lisses et des surfaces de del Pezzo,
comme on peut les trouver par exemple dans \cite{beauville} et \cite{manin}. On utilise en particulier le fait
qu'une surface cubique lisse   sur un corps $\k$ s\'eparablement clos est d\'eploy\'ee :
les 27 droites sont d\'efinies sur $\k$. 
Rappelons que deux droites de l'espace projectif
sont dites gauches (l'une \`a l'autre) si elles ne se rencontrent pas.
 
Nous remercions le rapporteur de sa lecture attentive.

\section{Surfaces cubiques lisses poss\'edant une droite rationnelle}\label{cubiquesavecdroite}

Commen\c cons par un lemme g\'en\'eral, sans doute classique.
\begin{lem}\label{2div}
 Soit $X$ une surface cubique lisse sur un corps $k$ s\'eparablement clos. 
Soit $\omega \in \Pic(X)$ la classe du faisceau canonique.
Soient $5$ droites gauches $L_{i}$, $i=1,\dots,5$ sur $X$.
 
Les deux conditions suivantes sont \'equivalentes :

\begin{itemize}
	\item[(i)] Il existe une sixi\`eme droite gauche aux 5 droites.

	\item[(ii)] La somme $\sum_{i=1}^5 L_{i} -\omega $ n'est pas divisible par 2 dans $\Pic(X)$.
\end{itemize}
Si elles sont satisfaites, la sixi\`eme droite gauche est uniquement  d\'efinie.
\end{lem}

\begin{proof}
Soit $L_{6}$ une droite gauche aux $L_{i}$, $i=1,\dots,5$. En contractant tous les $L_{i}$, $i=1, \dots, 6$,
 on obtient le plan $\P^2$.
La classe canonique $\omega$ de $X$ est   $-3 \lambda_{0}+ \sum_{i=1}^6 L_{i}$, o\`u $\lambda_{0}$ est donn\'e par  l'image r\'eciproque d'une droite de $\P^2$ (c'est une cubique gauche dans $X \subset \P^3$).
Le groupe $\Pic(X)$ est le groupe ab\'elien libre sur les g\'en\'erateurs $\lambda_{0}, L_{1},�\dots,L_{6}$.
On a  $\sum_{i=1}^5 L_{i} -\omega = 3\lambda_{0} -L_{6}$, non divisible par 2 dans $\Pic(X)$.

S'il n'existe pas de sixi\`eme droite gauche aux 5 droites, alors en contractant ces 5 droites,
on obtient $\P^1 \times \P^1$. La classe canonique $\omega$ sur $X$ est alors  
$\omega= -2e_{1}-2e_{2}+ \sum_{i=1}^5 L_{i}$, o\`u $e_{1}$ et $e_{2}$ sont les images r\'eciproques des
g\'en\'eratrices de  $\P^1 \times \P^1$. Le groupe $\Pic(X)$ est le groupe ab\'elien libre sur les
$e_{1},e_{2}, L_{1}, \dots, L_{5}$. On a ici  $\sum_{i=1}^5 L_{i}-\omega = 2e_{1}+2e_{2}$ divisible par 2
dans $\Pic(X)$.

S'il existe une sixi\`eme droite gauche, on consid\`ere la contraction des 6 droites en 6 points $P_{1}, \dots, P_{6}$ en
position g\'en\'erale dans $\P^2$.
Les 27 droites de $X$ correspondent aux 6 points, aux 15 droites passant par deux des 6 points et aux 6 coniques par 5 des 6 points.
On voit ais\'ement  sur ce mod\`ele que la droite de $X$ correspondant \`a $P_{6}$ est la seule droite gauche aux
droites $L_{i}$, $i=1,\dots,5$.
\end{proof}

\begin{prop}\label{generalfibre}
Soient $k$ un corps et $X$ une $k$-surface cubique lisse poss\'edant une $k$-droite $L$.
 Soit $f : X \to \P^1_{k}$ la fibration en coniques d\'efinie par la famille
 des 2-plans contenant la $k$-droite $L$.
 
 (i)  Cette fibration poss\`ede 5 fibres g\'eom\'etriques r\'eductibles, chacune
 compos\'ee de 2 droites concourantes de $\X$.
 
 (ii) Le morphisme induit $L \to \P^1_{k}$  induit par $f$ est de degr\'e 2.
 
 (iii) La fibre g\'en\'erique $X_{\eta}$ est une conique lisse sur le corps $k(\P^1)$
 qui poss\`ede un point de degr\'e 2 d\'efini par la restriction de $L$.
\end{prop}
 \begin{proof}  
 La structure des mauvaises fibres g\'eom\'etriques de $f$ est bien connue
 \cite[Lemme IV.15]{beauville}. Pour la classe d'une fibre $F$ on a $L \cdot F=2$.
  \end{proof}

 Le th\'eor\`eme qui suit a un air \guillemotleft \,classique \guillemotright,
 mais ne semble pas se trouver dans la litt\'erature. 

 \begin{theo} \label{sectiondroite}
Soient $k$ un corps et $X$ une $k$-surface cubique lisse poss\'edant une $k$-droite $L$.
 Soit $f : X \to \P^1_{k}$ la fibration en coniques d\'efinie par la famille
 des 2-plans contenant la $k$-droite $L$.  Alors $f$ poss\`ede une section
 si et seulement s'il existe une $k$-droite $L' \subset X$ qui ne rencontre pas $L$.
Toute telle droite est une section de $f$.
  \end{theo}  
\begin{proof}
Une direction est facile et bien connue.
 Soit $F$ la classe d'une fibre  sur un $k$-point. La classe du diviseur $L + F$ est la classe anticanonique, c'est-\`a-dire la classe d'une section plane de $X$.
Soit  $L' \subset X$  une  $k$-droite qui ne rencontre pas $L$. De  $((L + F) \cdot L') = 1$ et
  $L \cdot L' = 0$ on d\'eduit $L' \cdot F = 1$. Ainsi   $L'$ d\'efinit une section de $f$.

 Pour l'autre direction, supposons donn\'ee $s$ une section de $f$, identifi\'ee \`a son image dans $X$.    Consid\'erons l'ensemble $\mathfrak{L}$ des
$10$ droites de $\X$ qui rencontrent $L$.
Comme $\X$ est lisse, la section $s$ rencontre chaque fibre de $f$ en un point lisse
de la fibre.
En particulier, elle rencontre chaque fibre singuli\`ere  de $\overline{f}=f\times_{k}\k$ en un point d'une seule des deux $\k$-droites formant la fibre.
On trouve que $\mathfrak{L}$ est la r\'eunion de  deux  ensembles Galois-invariants $\mathfrak{L}_1$ et $\mathfrak{L}_2$, 
qui chacun consiste de $5$ droites qui ne se rencontrent pas deux \`a deux. 
La somme $\mathfrak{L}_1 + \mathfrak{L}_2$ a pour classe $5F \in \Pic(\X)$, o\`u $F$ d\'esigne une fibre
de $X \to \P^1_{\k} $ au-dessus d'un $\k$-point. Le nombre d'intersection $((\mathfrak{L}_1 + \mathfrak{L}_2) \cdot s)$
est donc \'egal \`a $5$.   
Mais alors on ne peut avoir \`a la fois $\mathfrak{L}_1-\omega$ et $ \mathfrak{L}_2-\omega$ divisibles par $2$
dans $\Pic(\X)$, sinon $\mathfrak{L}_1 + \mathfrak{L}_2$ serait divisible par 2 dans $\Pic(\X)$. Ainsi, d'apr\`es le lemme \ref{2div}, pour $j=1$ ou $j=2$, il existe   une droite $L'$  sur $\X$
qui ne rencontre aucune des droites de $\mathfrak{L}_j$, et cette droite est unique et donc d\'efinie sur $k$.
Elle est distincte des droites de $\mathfrak{L}_1$ et $\mathfrak{L}_2$, qui sont toutes les droites rencontrant $L$,
donc elle ne rencontre pas $L$.
\end{proof}

\begin{prop}\label{generalfibre2}
Soient $k$ un corps et $X$ une $k$-surface cubique lisse poss\'edant une $k$-droite $L$.
 Soit $f : X \to \P^1_{k}$ la fibration en coniques d\'efinie par la famille
 des 2-plans contenant la $k$-droite $L$. Soit $O \in \P^1(k)$ un $k$-point et $X_{O}$ la fibre. 
 Soit $I$ l'ensemble des composantes
 irr\'eductibles des fibres $X_{P}$, pour $P$ un point ferm\'e de $\P^1_{k}$ dont la fibre
 $X_{P}$ n'est pas g\'eom\'etriquement int\`egre sur le corps r\'esiduel $k(P)$. Notons $C_{i}, i\in I,$ ces
 diverses composantes. 
 
 (i)  La restriction \`a la fibre g\'en\'erique donne naissance \`a une suite
 exacte $$ \Z X_{0} \oplus_{i \in I} \Z C_{i} \to \Pic(X) \to \Pic(X_{\eta}) \to 0.$$
 
 (ii) La classe de $L$ et la classe de chacune des  courbes $C_{i}$ appartient au sous-groupe 
$ \Delta \subset \Pic(X)$.  

(iii) Le groupe $\Pic(X)/\Delta$ est un quotient de $ \Z/5$, et est engendr\'e par la classe de $X_{O}$.
\end{prop}
\begin{proof} Le point $O$ est ici un  point quelconque de  $\P^1(k)$, on n'a pas besoin 
de supposer la fibre $X_{O}$ lisse sur $k$.
L'\'enonc\'e (i) est clair, car le  diviseur d\'efini  par
une  fibre au-dessus d'un point ferm\'e  de $\P^1_{k}$ a sa classe dans
le sous-groupe engendr\'e par   $X_{O}$ dans $\Pic(X)$. Dans $\Pic(X)$, la classe de $5X_{O}$ est \'egale \`a
 la  somme des fibres d\'eg\'en\'er\'ees de $f$, et chacune de ces fibres appartient \`a $\Delta$.
  Notons $\omega= \omega_{X_{\eta}/k(\P^1)}$
 la classe du faisceau canonique. 
 Le groupe $\Pic(X_{\eta})$ est libre de rang 1.
 
  On a
 $\Pic(X_{\eta})=\Z  \omega$, induit par l'image de la classe de $L$, si et seulement si la conique $X_{\eta}/k(\P^1)$ 
 n'a pas de point rationnel, c'est-\`a-dire si et seulement si la fibration $f$ n'admet pas de section. 
Dans ce cas les \'enonc\'es (i) et (ii) donnent   (iii).

Si la fibration $ f$ poss\`ede une section, alors d'apr\`es le th\'eor\`eme \ref{sectiondroite} il existe une
$k$-droite $L'$ section de la fibration et   ne rencontrant pas $L$.
Cette $k$-droite $L'$ est dans $\Delta$ et son image engendre $\Pic(X_{\eta})$.  
Les \'enonc\'es (i) et (ii) donnent alors (iii).
\end{proof}

\begin{rema}
	Dans les articles \cite{SSD} et \cite{freisofos}, les auteurs \'etablissent la bonne
	borne inf\'erieure pour la conjecture de Manin \cite{FMT89} concernant les points rationnels de hauteur
	born\'ee pour les surfaces cubiques lisses poss\'edant deux $k$-droites gauches
	(sur $\Q$ et sur n'importe quel corps de nombres $k$, respectivement). 
	Le th\'eor\`eme \ref{sectiondroite} montre que ces r\'esultats   valent sous l'hypoth\`ese
	(\'equivalente, mais a priori plus faible) qu'il y a une $k$-droite dont la fibration en coniques
	associ\'ee admet une section.
\end{rema}

\section{Le quotient $\Pic(X)/\Delta$ pour une surface cubique lisse}\label{toutescubiques}

Soient $k$ un corps
et  $X \subset \P^3_{k}$ une $k$-surface cubique projective lisse. 
Notons $\Delta=\Delta(X)$ le sous-groupe de $\Pic(X)$   engendr\'e
par les classes de normes de $K$-droites sur $X_{K}$ pour 
toutes les extensions finies s\'eparables de corps $K/k$.

Si la surface $X$ est d\'eploy\'ee sur $k$, c'est-\`a-dire si 
les 27 droites de $X$ sont d\'efinies sur $k$, alors $\Delta(X)=\Pic(X)$.

Nous laissons au lecteur la d\'emonstration du lemme suivant.

\begin{lem} \label{normres}
Soit $K/k$ une extension finie s\'eparable de corps.

(a) La restriction  $\Pic(X) \to \Pic(X_{K})$ induit un homomorphisme
$$\Res_{k,K} : \Pic(X)/\Delta(X) \to \Pic(X_{K})/\Delta(X_{K}).$$

(b) La norme $\Pic(X_{K}) \to \Pic(X)$ induit un homomorphisme
$$\Norm_{K/k} :  \Pic(X_{K})/\Delta(X_{K}) \to \Pic(X)/\Delta(X).$$

(c) Le compos\'e $\Norm_{K/k} \circ \Res_{k,K} $ est la multiplication
par le degr\'e $[K:k]$.

(d) Si la $K$-surface cubique $X_{K}$ est d\'eploy\'ee, alors
$\Pic(X)/\Delta(X)$ est un groupe fini annul\'e par le degr\'e $[K:k]$. $\Box$
\end{lem}
 
\begin{prop}\label{bonl}
 Soient $k$ un corps
et  $X \subset \P^3_{k}$ une $k$-surface cubique projective lisse.  Soit $\ell$ un nombre premier.
Le groupe fini $\Pic(X)/\Delta$
n'a pas de $\ell$-torsion pour $\ell\neq 3, 5$.
\end{prop}

\begin{proof}
Soit $K/k$ une extension galoisienne finie, de groupe $G$, d\'eployant la surface $X$,
par exemple la plus petite extension sur laquelle les 27 droites sont d\'efinies.
Soit $G_{\ell}$ un $\ell$-sous-groupe de Sylow de $G$ et soit $K_{\ell } \subset K$
le corps fixe. Il r\'esulte du lemme  \ref{normres}  que le noyau de
$$\Res_{k,K_{l}}: \Pic(X)/\Delta(X) \to  \Pic(X_{K_{\ell}})/\Delta(X_{K_{\ell}})$$
est annul\'e par un entier premier \`a $\ell$, donc est injectif sur la torsion $\ell$-primaire.
Pour \'etablir la proposition, on peut donc supposer que $G$ est un groupe fini
$\ell$-primaire, et, toujours d'apr\`es le lemme \ref{normres} , que $ \Pic(X)/\Delta(X)$ est un groupe fini
$\ell$-primaire.
  Comme $\ell \neq 3$, et qu'il y a 27 droites sur $\X$,
  il existe alors une droite sur $X$. D'apr\`es la proposition
\ref{generalfibre2}, le quotient $\Pic(X)/\Delta$ est annul\'e par $5$.
Comme on a suppos\'e $\ell \neq 5$, on conclut que $\Pic(X)/\Delta$ est nul.
  \end{proof}

\begin{prop}\label{l5} 
Soient $k$ un corps
et  $X \subset \P^3_{k}$ une $k$-surface cubique projective lisse.  Le groupe fini $\Pic(X)/\Delta$
n'a pas de 5-torsion.
\end{prop}
\begin{proof}
Par un argument de norme similaire \`a celui donn\'e dans la proposition \ref{bonl}, on peut supposer que $X$ est d\'eploy\'ee par une extension finie galoisienne $K/k$ de degr\'e une puissance de 5.
  
Comme $27$ n'est pas divisible par 5,
$X$ contient  une $k$-droite $L$.
En utilisant une telle $k$-droite, on d\'efinit une fibration en coniques
$f : X \to \P^1_{k}$. Comme $X_{K}$ est une surface cubique d\'eploy\'ee,
il existe une $K$-droite qui ne rencontre pas la $K$-droite $L_{K}$.
D'apr\`es le th\'eor\`eme  \ref{sectiondroite} appliqu\'e au niveau de $K$, une telle $K$-droite est une section de
$f_{K}$. La fibre g\'en\'erique   $X_{\eta}/k(\P^1)$ est donc une conique
lisse qui poss\`ede un point sur l'extension $K(\P^1)/k(\P^1)$, de degr\'e une puissance de 5, et en particulier de degr\'e impair. Par un th\'eor\`eme bien connu, ceci implique que la conique
 $X_{\eta}/k(\P^1)$ a un $k(\P^1)$-point rationnel. La fibration $f : X \to \P^1_{k}$
 poss\`ede donc une section sur $k$. Appliquant    le th\'eor\`eme  \ref{sectiondroite} au niveau de $k$,
 on obtient l'existence d'une $k$-droite $L' \subset X$ qui est une section de la fibration et ne rencontre pas $L$.

 Si on prend la fibration associ\'ee \`a $L$, on voit qu'elle a 5 fibres g\'eom\'etriques d\'eg\'en\'er\'ees, 
et que toutes leurs composantes soit sont d\'efinies sur $k$, soit 
sont d\'efinies sur la m\^eme extension cyclique
 de degr\'e 5 de $k$.

Si les composantes sont toutes d\'efinies sur $k$, alors
 $\Pic(X)=\Pic(\X)$ et toutes les droites de $X$ sont d\'efinies sur $k$.
On a donc alors $ \Pic(X)/\Delta=0$.  

Supposons que les composantes ne sont pas d\'efinies sur $k$.
 La droite $L'$ rencontre une seule fois chaque fibre g\'eom\'etrique d\'eg\'en\'er\'ee.
Dans chaque fibre, on peut donc choisir la composante qui ne rencontre pas $L'$.
La r\'eunion de ces  5 composantes $R_{i}$ (qui forment une orbite du groupe de Galois) 
et de $L'$ forme un sextuplet de droites
gauches deux \`a deux, globalement rationnel. On peut contracter ce
sextuplet sur un $\P^2_{k}$. Dans le plan, consid\'erons la conique qui contient
les 5 points conjugu\'es images des $R_{i}$. Elle est d\'efinie sur $k$. Son image inverse
dans $X$ est une $k$-droite, dont la classe est $2 \lambda_{0} - \sum_{i=1}^5 R_{i}$,
o\`u l'on a not\'e $\lambda_{0} \in \Pic(X)$ la classe de l'image r\'eciproque d'une $k$-droite de $\P^2$.
On voit donc que $2\lambda_{0}$  est dans $\Delta$.  Le groupe de Picard de 
$\overline X$ est la somme directe $\Z \lambda_{0} \oplus \Z L' \oplus \sum_{i=1}^5 \Z R_{i}$.
Le groupe de Picard de $X$ est donc engendr\'e par $\lambda_{0}$, $L'$ et
 la somme $ \sum_{i=1}^5 R_{i}$. On voit donc que  $\Pic(X)/\Delta$ est annul\'e par 2.
 Comme il est annul\'e par 5 (la proposition \ref{generalfibre2} s'applique, car $X$ poss\`ede une $k$-droite), il est nul.
 \end{proof}

La combinaison des propositions \ref{generalfibre2}, \ref{bonl} et  \ref{l5}  
 donne le th\'eor\`eme \ref{thmgeneral}.

\begin{rema} On ne peut omettre l'hypoth\`ese d'existence d'une droite $k$-rationnelle
dans le th\'eor\`eme \ref{thmgeneral} (ii).
Soit $K/k$ une extension cubique cyclique de corps.
Prenons dans $\P^2_{k}$ deux points ferm\'es $M_{1}$ et $M_{2}$ de degr\'e 3 en position g\'en\'erale,
chacun d\'efini sur $K$.
Sur la surface cubique $X \subset \P^3_{k}$ obtenue par \'eclatement de $M_{1}$ et $M_{2}$, dont on sait bien d\'ecrire les 27 droites g\'eom\'etriques
en termes des 6 points de $\P^2_{\k}$ d\'efinis par $M_{1} \cup M_{2}$, des droites  passant par 2 de ces 6 points, et des coniques passant par 5 de ces 6 points,
on v\'erifie que 
 la classe $\lambda_{0} \in \Pic(X)$  de l'image r\'eciproque d'une $k$-droite de $\P^2_{k}$ n'est pas dans $\Delta \subset \Pic(X)$, et l'on a $3\lambda_{0} \in \Delta$.
\end{rema}

On a le corollaire :
 
\begin{cor}
Soient $k$ un corps
et  $X \subset \P^3_{k}$ une $k$-surface cubique projective lisse
 poss\'edant une $k$-droite.
Soit $U$ le compl\'ementaire des droites g\'eom\'etriques et ${\overline k}[U]^*$ le groupe des
fonctions rationnelles sur $\X$ inversibles sur $\overline{U}$.
Alors le r\'eseau $\hat{T}:={\overline k}[U]^*/{\overline k}^*$ est un module galoisien coflasque :
pour tout sous-groupe ferm\'e $H$ du groupe de Galois absolu $G$ 
de $k$, 
on a
$H^1(H, \hat{T})=0$.
La suite exacte
$$0 \to {\overline k}[U]^*/{\overline k}^* \to \oplus_{i=1}^{27} \Z l_{i} \to \Pic({\overline X}) \to 0,$$
o\`u les $l_{i}$  sont les 27 droites de $\overline{X}$, et le module galoisien 
$\oplus_{i=1}^{27} \Z l_{i}$ est un module de permutation,
est une r\'esolution coflasque du module galoisien $\Pic({\overline X})$.
\end{cor}

\begin{proof}
 Lorsque l'on prend les points fixes sous l'action du groupe $G$, on
obtient la suite exacte
$$(\oplus_{i=1}^{27} \Z l_{i})^G \to  \Pic({\overline X})^G \to H^1(G,{\overline k}[U]^*/{\overline k}^* ) \to  H^1(G,\oplus_{i=1}^{27} \Z l_{i}).$$

On a $\Pic(X) =  \Pic(\X)^G$ car $X$ poss\`ede un $k$-point.
C'est un r\'esultat tr\`es bien connu, qu'on peut par exemple \'etablir ainsi. Pour toute $k$-vari\'et\'e projective et lisse g\'eom\'etriquement int\`egre, on a  une suite exacte
$$0 \to \Pic(X) \to  \Pic(\X)^G \to \Br(k) \to \Br(X)$$
fournie par la suite spectrale de Leray pour la projection $X \to \Spec(k)$ et le faisceau \'etale
$\G_{m,X}$ sur $X$. Un point $k$-rationnel fournit une section de $\Br(k) \to \Br(X)$.

L'image du groupe $(\oplus_{i=1}^{27} \Z l_{i})^G$ dans $\Pic(X) =  \Pic(\X)^G$ est 
le sous-groupe $\Delta$, \'egal \`a $\Pic(X) $ d'apr\`es le th\'eor\`eme  \ref{thmgeneral} (ii).
Le module galoisien $ \oplus_{i=1}^{27} \Z l_{i}$ est un $G$-module de permutation,
donc $H^1(G,\oplus_{i=1}^{27} \Z l_{i})=0$  
(lemme de Shapiro et nullit\'e des groupes $H^1(H,\Z)$ pour les sous-groupes ferm\'es de $G$).
On a donc bien $H^1(G,{\overline k}[U]^*/{\overline k}^*)=0$.
Le m\^{e}me argument donne  $H^1(H,{\overline k}[U]^*/{\overline k}^*)=0$ 
pour
tout sous-groupe ferm\'e $H$ de $G$.
 \end{proof}
 \begin{rema}
 Le $G$-module $\Pic(\X)$ est auto-dual, via la forme d'intersection.
On obtient donc  une r\'esolution flasque de $\Pic(\X)$ en prenant la suite duale de celle de l'\'enonc\'e via $\Hom_{\Z}(., \Z)$.
\end{rema}

\end{document}